\documentclass[12pt]{article}
\usepackage[centertags]{amsmath}
\usepackage{amsfonts}
\usepackage{amssymb}
\usepackage{amsthm}

\theoremstyle{plain}
\newtheorem{thm}{Theorem}[section]
\newtheorem{cor}[thm]{Corollary}

\newtheorem{prop}[thm]{Proposition}

\theoremstyle{definition}

\theoremstyle{remark}
\numberwithin{equation}{section}

\newcommand{\beast}{\begin{eqnarray*}}
\newcommand{\eeast}{\end{eqnarray*}}
\title{Spectral Radius and Degree Sequence\\ of a Graph}

\author{Chia-an Liu\footnote{Corresponding author. E-mail address: twister.imm96g@g2.nctu.edu.tw (C.-A Liu).}
~\footnote{Department of Applied Mathematics, National Chiao Tung University, Taiwan R.O.C..}
\and Chih-wen Weng$^\dag$}
\date{August 9, 2012}

\begin{document}
\maketitle

\bibliographystyle{plain}

%%%%%%%%%%%  The TEX begins here %%%%%%%%%%%%%%%%%%%%%%%%%%%%%%%

\bigskip

\begin{abstract}
Let $G$ be a simple connected graph of order $n$
with degree sequence $d_{1}, d_{2}, \cdots, d_{n}$ in non-increasing order.
The \emph{spectral radius} $\rho(G)$ of $G$ is the largest eigenvalue of its adjacency matrix.
For each positive integer $\ell$ at most $n,$
we give a sharp upper bound for $\rho(G)$ by a function of
$d_{1}, d_{2}, \cdots, d_{\ell},$ which generalizes a series of previous results.

\bigskip

\end{abstract}

%%%%%%%%%%%%%%%%   1.Introduction  %%%%%%%%%%%%%%%%%%%%%%%%%%%
\section{Introduction}   \label{s1}

Let $G$ be a simple connected graph of $n$ vertices and $m$ edges with degree sequence
$d_{1} \geq d_{2} \geq \cdots \geq d_{n}.$ The {\it adjacency matrix} $A=(a_{ij})$ of $G$ is a binary square matrix of order $n$ with rows and columns indexed by the vertex set $VG$ of $G$ such that for any $i, j\in VG$, $a_{ij}=1$ if $i,j$ are adjacent in $G.$
The \emph{spectral radius} $\rho(G)$ of $G$ is the largest eigenvalue of its adjacency matrix,
which has been studied by many authors.

\medskip

The following theorem is well-known \cite[Chapter 2]{m:88}.
\begin{thm}   \label{thm1.1}
If $A$ is a nonnegative irreducible $n \times n$ matrix with largest eigenvalue $\rho(A)$ and row-sums
$r_{1}, r_{2},\ldots,r_{n},$ then
\begin{equation}
\rho(A) \leq \max_{1 \leq i \leq n}r_{i}   \nonumber
\end{equation}
with equality if and only if the row-sums of $A$ are all equal.
\end{thm}

\medskip

In 1985 \cite[Corollary~2.3]{bh:85}, Brauldi and Hoffman showed the following result.
\begin{thm}   \label{thm1.2}
If $m \leq k(k-1)/2,$ then
\begin{equation}
\rho(G) \leq k-1   \nonumber
\end{equation}
with equality if and only if $G$ is isomorphic to
the complete graph $K_{n}$ of order $n.$
\end{thm}

\medskip

In 1987 \cite{s:87}, Stanley improved Theorem~\ref{thm1.2} and showed the following result.
\begin{thm}   \label{thm1.3}
\begin{equation}
\rho(G) \leq \frac{-1+\sqrt{1+8m}}{2}   \nonumber
\end{equation}
with equality if and only if $G$ is isomorphic to
the complete graph $K_{n}$ of order $n.$
\end{thm}

\medskip

In 1998 \cite[Theorem~2]{h:98}, Yuan Hong improved  Theorem~\ref{thm1.3} and showed the following result.
\begin{thm}   \label{thm_genus}
\begin{equation}
\rho(G) \leq \sqrt{2m-n+1}   \nonumber
\end{equation}
with equality if and only if $G$ is isomorphic to the star $K_{1,n-1}$ or
to the complete graph $K_{n}.$
\end{thm}

\medskip

In 2001 \cite[Theorem~2.3]{hsf:01}, Hong et al. improved Theorem~\ref{thm_genus} and showed the following result.
\begin{thm}   \label{thm1.4}
\begin{equation}
\rho(G) \leq \frac{d_{n}-1+\sqrt{(d_{n}+1)^2 + 4(2m-n d_{n})}}{2}   \nonumber
\end{equation}
with equality if and only if $G$ is regular
or there exists $2 \leq t \leq n$ such that $d_{1}=d_{t-1}=n-1$ and $d_{t}=d_{n}.$
\end{thm}

\medskip

%In 2011 \cite[Theorem~2.3]{d:11}, Kinkar Ch. Das also gave an upper bound for the spectral radius
%as a part of the main results in that paper.
%\begin{thm}   \label{thm_kcdas}
%\begin{equation}
%\rho(G) \leq \frac{d_{2}-1+\sqrt{(d_{2}+1)^{2}+4(d_{1}-d_{2})}}{2}
%\nonumber
%\end{equation}
%with equality if and only if $G$ is regular, or $d_{1}=n-1$ and $d_{2}=d_{n}.$
%\end{thm}

%\medskip

In 2004 \cite[Theorem~2.2]{sw:04}, Jinlong Shu and Yarong Wu improved Theorem~\ref{thm1.1}
in the case that $A$ is the adjacency matrix of $G$ by showing the following result.
\begin{thm}   \label{thm1.5}
For $1 \leq \ell \leq n,$
\begin{equation}
\rho(G) \leq \frac{d_{\ell}-1+\sqrt{(d_{\ell}+1)^{2}+4(\ell-1)(d_{1}-d_{\ell})}}{2}   \nonumber
\end{equation}
with equality if and only if $G$ is regular
or there exists $2 \leq t \leq \ell$ such that $d_{1}=d_{t-1}=n-1$ and $d_{t}=d_{n}.$
\end{thm}
Moreover, they also showed in \cite[Theorem~2.5]{sw:04} that if $p+q \geq d_{1}+1$
then Theorem~\ref{thm1.5} improves Theorem~\ref{thm1.4}
where $p$ is the number of vertices with the largest degree $d_{1}$
and $q$ is the number of vertices with the second largest degree.
The special case $\ell=2$ of Theorem~\ref{thm1.5} is reproved \cite{d:11}.

\medskip

In this research, we present a sharp upper bound of $\rho(G)$ in terms of the degree sequence of $G,$
which improves all the above theorems.

%\begin{defn}    \label{defn1.6}
%For $1 \leq \ell \leq n,$ let
%\begin{equation}
%\phi_{\ell} = \frac{d_{\ell}-1+\sqrt{(d_{\ell}+1)^2+4\sum_{i=1}^{\ell-1}(d_{i} - d_{\ell})}}{2} \geq d_{\ell}.   \nonumber
%\end{equation}
%\end{defn}

\medskip

\begin{thm}   \label{thm1.6}
For $1 \leq \ell \leq n,$
$$\rho(G) \leq \phi_{\ell}:= \frac{d_{\ell}-1+\sqrt{(d_{\ell}+1)^2+4\sum_{i=1}^{\ell-1}(d_{i} - d_{\ell})}}{2},$$
with equality if and only if $G$ is regular
or there exists $2 \leq t \leq \ell$ such that $d_{1}=d_{t-1}=n-1$ and $d_{t}=d_{n}.$
\end{thm}

This result improves Theorem~\ref{thm1.4} and Theorem~\ref{thm1.5}
since $\phi_{n}$ is exactly the upper bounds in Theorem~\ref{thm1.4} and  is at most the upper bound appearing in Theorem~\ref{thm1.5}.

\medskip

Note that the number $\phi_\ell$ defined in Theorem~\ref{thm1.6} is at least $d_\ell.$
The sequence $\phi_{1}, \phi_{2}, \cdots, \phi_n$ is not necessary to be non-increasing. %decreasing (in a non-strict sense).
We show that this sequence is first non-increasing and then non-decreasing, and determine its lowest value in Section~\ref{s4}.

\section{Proof of Theorem~\ref{thm1.6}}   \label{s3}
\begin{proof}
Let the vertices be labeled by $1,2,\ldots,n$ with degrees $d_{1} \geq d_{2} \geq \cdots \geq d_{n},$ respectively.
For each $1 \leq i \leq \ell-1,$ let $x_{i} \geq 1$ be a variable to be determined later.
Let $U=diag(x_{1},x_{2},\ldots,x_{\ell-1},1,1,\ldots,1)$ be a diagonal matrix of size $n \times n$.
Then $U^{-1}=diag(x_1^{-1}, x_2^{-1},\ldots,x_{\ell-1}^{-1},1,1,\ldots,1).$
% define B
Let $B=U^{-1}AU.$
Note that $A$ and $B$ have the same eigenvalues.
% introduce row-sums
Let $r_{1},r_{2},\ldots,r_{n}$ be the row-sums of $B.$
Then for $1 \leq i \leq \ell-1$ we have
\begin{eqnarray}
r_{i} &=& \sum_{k=1}^{\ell-1}\frac{x_{k}}{x_{i}}a_{ik} + \sum_{k=\ell}^{n}\frac{1}{x_{i}}a_{ik}   \nonumber
= \frac{1}{x_{i}}\sum_{k=1}^{n}a_{ik} + \frac{1}{x_{i}}\sum_{k=1}^{\ell-1}(x_{k}-1)a_{ik}\\
&\leq& \frac{1}{x_{i}}d_{i}+\frac{1}{x_{i}} \left( \sum_{k=1,k \neq i}^{\ell-1}x_{k}-(\ell-2) \right),   \label{eq3.1}
\end{eqnarray}
and for $\ell \leq j \leq n$ we have
\begin{eqnarray}
r_{j} &=& \sum_{k=1}^{\ell-1}x_{k}a_{jk} + \sum_{k=\ell}^{n}a_{ik}   \nonumber
=\sum_{k=1}^{n}a_{jk} + \sum_{k=1}^{\ell-1}(x_{k}-1)a_{jk}\\
&\leq& d_{\ell} + \left( \sum_{k=1}^{\ell-1}x_{k}-(\ell-1) \right).   \label{eq3.2}
\end{eqnarray}
For $1 \leq i \leq \ell-1$ let
\begin{equation}
x_{i}=1+\frac{d_{i}-d_{\ell}}{\phi_{\ell}+1} \geq 1,   \label{eq3.4}
\end{equation}
where $\phi_{\ell}$ is defined in Theorem~\ref{thm1.6}.
Then for $1 \leq i \leq \ell-1$ we have
\begin{equation}
r_{i} \leq \frac{1}{x_{i}}d_{i}+\frac{1}{x_{i}} \left( \sum_{k=1,k \neq i}^{\ell-1}x_{k}-(\ell-2) \right) = \phi_{\ell},
\nonumber
\end{equation}
and for $\ell \leq j \leq n$ we have
\begin{equation}
r_{j} \leq d_{\ell} + \left( \sum_{k=1}^{\ell-1}x_{k}-(\ell-1) \right) = \phi_{\ell}.  \nonumber
\end{equation}
Hence by Theorem~\ref{thm1.1},
\begin{equation}
\rho(G) = \rho(B) \leq \max_{1 \leq i \leq n}\{r_{i}\} \leq \phi_{\ell}.   \label{eq3.5}
\end{equation}
The first part of Theorem~\ref{thm1.6} follows.

\medskip

%%%%%%%%%%  Sufficiency Begin %%%%%%%%%%%%
The sufficient condition of $\phi_{\ell}=\rho(G)$ follows from the fact that
\begin{equation}
\phi_{\ell} \leq \frac{d_{\ell}-1+\sqrt{(d_{\ell}+1)^{2}+4(\ell-1)(d_{1}-d_{\ell})}}{2}   \nonumber
\end{equation}
and applying the second part in Theorem~\ref{thm1.5}.
%%%%%%%%%%  Sufficiency End %%%%%%%%%%%%

\medskip

%%%%%%%%%%  Necessity Begin %%%%%%%%%%%%
To prove the necessary condition of $\phi_{\ell}=\rho(G)$, suppose $\phi_{\ell}=\rho(G).$
Then the equalities in~\eqref{eq3.1} and~\eqref{eq3.2} all holds.
If $d_{1}=d_{\ell},$ then $d_{1}=\phi_{1}=\phi_{\ell}=\rho(G),$ and $G$ is regular by the second part of Theorem~\ref{thm1.1}.
Suppose $2 \leq t \leq \ell$ such that $d_{t-1} > d_{t} = d_{\ell}.$
Then $x_{i} > 1$ for $1 \leq i \leq t-1$ by~\eqref{eq3.4}.
For each $1 \leq i \leq \ell-1,$ the equality in~\eqref{eq3.1} implies that $a_{ik}=1$ for $1 \leq k \leq t-1,$ $k \neq i.$
For each $\ell \leq j \leq n,$ the equality in~\eqref{eq3.2} implies that $a_{jk}=1$
for $1 \leq k \leq t-1$ and $d_{j}=d_{\ell}.$
Hence $n-1=d_{1}=d_{t-1} > d_{t}=d_{\ell}=d_{n}.$
%%%%%%%%%%  Necessity End %%%%%%%%%%%%

\medskip

We complete the proof.
\end{proof}

%%%%%%%%%%%%%%%%   3.More on $\phi_\ell$  %%%%%%%%%%%%%%%%%%%%%%%%%%%

\section{The sequence $\phi_{1},\phi_{2},\ldots,\phi_n$}   \label{s4}
The sequence $\phi_{1},\phi_{2},\ldots,\phi_{n}$ is not necessarily non-increasing. For example,
the path $P_{n}$ of $n$ vertices has $2=d_{1}=d_{n-2}>d_{n-1}=d_{n}=1,$
and it is immediate to check that
if $n \geq 6$ then $\phi_{1}=\phi_{2}=2 < \sqrt{n-1}=\phi_{n-1}=\phi_{n}.$
\medskip

Clearly that for all $1 \leq s < t \leq n,$ $d_{s}=d_{t}$ implies that $\phi_{s}=\phi_{t}.$
However, $\phi_{s}=\phi_{t}$ dose not imply $d_{s} = d_{t}.$
For example, in the graph with degree sequence $(4,3,3,2,1,1),$
one can check that $\phi_{4}=\phi_{5}=3$ but $d_{4} > d_{5}.$
\medskip

Recall that $d_{s}=d_{s+1}$ implies $\phi_{s}=\phi_{s+1}$ for $1 \leq s \leq n-1.$
The following proposition describes the shape of the sequence $\phi_1,$ $\phi_2$, $\ldots,$ $\phi_n.$

\begin{prop}   \label{prop4.1}
Suppose $d_{s}>d_{s+1}$ for $1 \leq s \leq n-1,$ and let $\succeq  \in \{ >, = \}.$
Then  $$\phi_{s} \succeq \phi_{s+1}~~ \hbox{iff~~} \sum_{i=1}^{s} d_{i} \succeq s(s-1).$$
\end{prop}
\begin{proof} Recall that
\begin{equation}
\phi_{s} = \frac{d_{s}-1+\sqrt{(d_{s}+1)^2+4\sum_{i=1}^{s-1}(d_{i} - d_{s})}}{2}. %\geq d_{\ell}
\nonumber
\end{equation} Consider the following equivalent relations step by step.
\begin{align*}
~~~&\phi_{s} \succeq \phi_{s+1}  \\
\Leftrightarrow~~~&
d_{s} - d_{s+1}+ \sqrt{(d_{s}+1)^{2}+4\sum_{i=1}^{s-1}(d_{i}-d_{s})} \\
 &~~~~~~~~~~~~\succeq\sqrt{(d_{s+1}+1)^{2}+4\sum_{i=1}^{s}(d_{i}-d_{s+1})}  \\
%%(d_{\ell} - d_{\ell+1})^2 + 2(d_{\ell} - d_{\ell+1})\sqrt{(d_{\ell}+1)^{2}+4\sum_{i=1}^{\ell}(d_{i}-d_{\ell})}
%\nonumber \\
%& + (d_{\ell}+1)^{2}+4\sum_{i=1}^{\ell}(d_{i}-d_{\ell}) \geq (d_{\ell+1}+1)^{2}+4\sum_{i=1}^{\ell}(d_{i}-d_{\ell+1})
%\nonumber \\
%\Leftrightarrow&
%(d_{\ell}-d_{\ell+1})\left( d_{\ell}+1+\sqrt{(d_{\ell}+1)^{2}+4\sum_{i=1}^{\ell}(d_{i}-d_{\ell})} \right) \geq 2\ell(d_{\ell}-d_{\ell+1})
%\nonumber \\
\Leftrightarrow~~~&
\sqrt{(d_{s}+1)^{2}+4\sum_{i=1}^{s-1}(d_{i}-d_{s})} \succeq 2s - (d_{s}+1)
\label{eq4.2} \\
\Leftrightarrow~~~&
(d_{s}+1)^{2}+4\sum_{i=1}^{s}(d_{i}-d_{s}) \succeq 4 s^{2} - 4 s(d_{s}+1) + (d_{s}+1)^{2}\\
\Leftrightarrow~~~&
\sum_{i=1}^{s}d_{i} \succeq s(s-1),
\end{align*}
where the third relation is obtained from the second  by taking square on both sides, simplifying it, and deleting the common term $d_{s}-d_{s+1}.$
\end{proof}

\medskip

%\begin{cor}  \label{cor4.2}
%Let $1 \leq \ell \leq n$ be the smallest integer such that $\sum_{i=1}^{\ell}d_{i} \leq \ell(\ell-1).$
%Then for any $1\leq i\leq n,$
%$$\phi_{i} = \min\{\phi_k~|~1\leq k\leq n\} ~~\hbox{iff~~} d_{i} = d_{\ell}.$$
%\end{cor}
%\begin{proof}
%\end{proof}

\begin{cor}  \label{cor4.2}
Let $3 \leq \ell \leq n$ be the smallest integer such that $\sum_{i=1}^{\ell}d_{i} < \ell(\ell-1).$
Then for $1 \leq j \leq n$ we have
$$\phi_{j} = \min\{\phi_k~|~1\leq k\leq n\}$$
if and only if $d_{j}=d_{\ell},$ or $d_{j}=d_{\ell-1}$ with $\sum_{i=1}^{\ell-1}d_{i} = (\ell-1)(\ell-2).$
\end{cor}
\begin{proof}
From Proposition~\ref{prop4.1},
$\sum_{i=1}^{\ell-1}d_{i} = (\ell-1)(\ell-2)$ implies $\phi_{\ell-1}=\phi_{\ell}.$
Also, clearly that $d_{j} = d_{\ell}$ implies $\phi_{j} = \phi_{\ell}.$
We show that $\phi_{\ell} = \min\{\phi_k~|~1\leq k\leq n\}$ in the following.

For $1 \leq s \leq \ell-1,$ from Proposition~\ref{prop4.1}
we have $\phi_{s} \geq \phi_{s+1}$ since $\sum_{i=1}^{s}d_{i} \geq s(s-1).$
For $\ell \leq t \leq n-1,$
note that $\sum_{i=1}^{t}d_{i} < t(t-1)$ implies $d_{t} < t-1,$
and hence $\sum_{i=1}^{t+1}d_{i} < t(t-1)+(t-1) < t(t+1).$
From Proposition~\ref{prop4.1} we have $\phi_{\ell} \leq \phi_{\ell+1} \leq \cdots \leq \phi_{n}$
since $\sum_{i=1}^{\ell}d_{i} < \ell(\ell-1).$
The result follows.
\end{proof}

%We immediately have the following corollary.
%\begin{cor}
%\begin{equation}
%\rho(G) \leq \min_{\ell \leq k \leq n} \left\{ \frac{d_{k}-1+\sqrt{(d_{k}+1)^2+4\sum_{i=1}^{k-1}(d_{i} - d_{k})}}{2} \right\},
%\nonumber
%\end{equation}
%where $\ell$ is the largest integer such that $d_{\ell} \geq \ell-1.$
%\end{cor}
%\begin{proof}
%It can be proved from Theorem~\ref{thm1.6} and Proposition~\ref{prop4.1}.
%\end{proof}

\section*{Acknowledgments}
This research is supported by the National Science Council of Taiwan R.O.C. under the project NSC 99-2115-M-009-005-MY3.

\bigskip

\noindent Chia-an Liu \hfil\break
Department of Applied Mathematics \hfil\break
National Chiao Tung University \hfil\break
1001 Ta Hsueh Road \hfil\break
Hsinchu, Taiwan 300, R.O.C. \hfil\break
Email: {\tt twister.imm96g@g2.nctu.edu.tw} \hfil\break
Ext: +886-3-5712121-56460 \hfil\break

\end{document}